\documentclass[12pt,twoside]{amsart}

\usepackage{amsmath, amssymb, amscd, paralist,tabularx,supertabular,amsmath, verbatim, amsthm, amssymb, mathrsfs, manfnt,times,latexsym,amscd,graphicx, color}
\usepackage{amsmath, amsthm, amssymb}
\usepackage{paralist}  
\usepackage{manfnt}    
\usepackage{url}
\usepackage[all]{xy}

\usepackage{fullpage}

\DeclareMathAlphabet{\curly}{U}{rsfs}{m}{n}

\DeclareMathOperator{\Br}{Br}

\DeclareMathOperator{\disc}{disc}

\DeclareMathOperator{\Ram}{Ram}

\DeclareMathOperator{\GL}{GL}

\DeclareMathOperator{\ind}{ind}

\DeclareMathOperator{\inv}{inv}

\DeclareMathOperator{\Res}{Res}

\DeclareMathOperator{\SL}{SL}

\DeclareMathOperator{\MM}{M}

\newtheorem{thm}{Theorem}[section]
\newtheorem{cor}[thm]{Corollary}

\newtheorem{prop}[thm]{Proposition}

\theoremstyle{definition}

\theoremstyle{remark}

\def\1{\mathbf{1}}
\def\disc{\mathrm{disc}}

\theoremstyle{remark}

\newtheorem*{ack}{Acknowledgements}

\theoremstyle{plain}

\def\1{\mathbf{1}}
\def\disc{\mathrm{disc}}

\newcommand{\frakp}{\mathfrak{p}}
\newcommand{\frakq}{\mathfrak{q}}

\newcommand{\frakP}{\mathfrak{P}}
\newcommand{\frakQ}{\mathfrak{Q}}

\newcommand{\abs}[1]{\left\vert#1\right\vert}

\newcommand{\pr}[1]{\left( #1 \right) }

\title{Brauer equivalent number fields and the geometry of quaternionic Shimura varieties}

\author{Benjamin Linowitz}
\address{Department of Mathematics\\ 
10 North Professor Street\\
Oberlin, OH 44074}
\email{benjamin.linowitz@oberlin.edu}

\begin{document}

\begin{abstract} 
Two number fields are said to be Brauer equivalent if there is an isomorphism between their Brauer groups that commutes with restriction. In this paper we prove a variety of number theoretic results about Brauer equivalent number fields (e.g., they must have the same signature). These results are then applied to the geometry of certain arithmetic locally symmetric spaces. As an example, we construct incommensurable arithmetic locally symmetric spaces containing exactly the same set of proper immersed totally geodesic surfaces.
\end{abstract}

\maketitle

\section{Introduction}

Given a Riemannian manifold $M$ it is natural to ask the extent to which the geometry of $M$ is dependent on the geodesics or totally geodesic submanifolds of $M$. This problem has received a considerable amount of attention in the case that $M$ is an arithmetic manifold. With respect to geodesics, it is a theorem of Reid \cite{R} that two arithmetic hyperbolic surfaces with the same geodesic length spectra are necessarily commensurable. This result was later generalized by Chinburg, Hamilton, Long and Reid \cite{CHLR} to hyperbolic $3$-manifolds, and by Prasad and Rapinchuk \cite{PR} to a very broad class of locally symmetric spaces.

More recently, McReynolds and Reid \cite{McReid} have proven that if two arithmetic hyperbolic $3$-manifolds have the same (nonempty) set of totally geodesic surfaces (up to commensurability) then they must in fact be commensurable. In a followup paper, McReynolds \cite{McR} showed that this behavior does not persist when one considers arithmetic manifolds with universal cover the symmetric space of products of $\SL_n(\mathbb R)$ and $\SL_n(\mathbb C)$. In particular McReynolds constructed examples of incommensurable arithmetic manifolds with precisely the same commensurability classes of totally geodesic submanifolds arising from a fixed field. McReynolds' proof was largely algebraic and reduced to constructing central simple algebras defined over number fields which were not isomorphic yet nevertheless contained precisely the same isomorphism classes of subalgebras defined over a fixed field.

To make all of this concrete, let $k$ be a number field with ring of integers $\mathcal O_k$ and $A$ be a central simple algebra over $k$ so that \[A\otimes_\mathbb Q \mathbb R \cong \MM_n(\mathbb R)^a \times \MM_n(\mathbb C)^b \times \MM_{n/2}(\mathbb H)^c\] for nonnegative integers $a,b,c$. Let $\mathcal O$ be an $\mathcal O_k$-order of $A$ and $\mathcal O^1$ the multiplicative subgroup consisting of those elements of $\mathcal O$ with reduced norm $1$. From the above isomorphism we obtain an embedding of $\mathcal O^1$ into $G_{a,b}=\SL_n(\mathbb R)^a\times \SL_n(\mathbb C)^b$. Let $\Gamma$ denote the image of $\mathcal O^1$ under this embedding. Then $\Gamma$ is an arithmetic lattice in $G_{a,b}$ and we obtain a locally symmetric space $X_{a,b}/\Gamma$, where $X_{a,b}$ is the symmetric space associated to $G_{a,b}$. We note that the commensurability class of $X_{a,b}/\Gamma$ is given by the isomorphism class of $A$. In light of this construction it is clear that in order to construct arithmetic manifolds with precisely the same commensurability classes of totally geodesic submanifolds arising from a fixed field it suffices to construct central simple algebras $A_1$ over $k_1$ and $A_2$ over $k_2$ with the property that for some field $F\subseteq k_1\cap k_2$, a central simple algebra $B$ over $F$ satisfies $B\otimes_F k_1\cong A_1$ if and only if $B\otimes_F k_2\cong A_2$. Much of McReynolds' paper was devoted to constructing algebras with this property.

In a recent paper, the author, McReynolds and Miller \cite{LMM} showed that one can construct algebras $A_1$ and $A_2$ with the aforementioned properties from any pair of number fields with isomorphic adele rings. Before describing this constriction we need to establish some notation. Given a number field $K$ we denote the Brauer group of $K$ by $\Br(K)$ and note that associated to any subfield $F\subset K$ there is a restriction homomorphism $\Res_{K/F}([B])=[B\otimes_F K]$. Suppose now that $k_1$ and $k_2$ are number fields whose adele rings are isomorphic. In \cite{LMM} it was shown that there is an isomorphism $\Phi_{\Br}: \Br(k_1)\rightarrow \Br(k_2)$ with the property that for any field $F\subset k_1\cap k_2$ and $L\supset k_1,k_2$ the following diagram commutes:
\begin{equation}\label{eqnBrauerNatural} 
\xymatrixrowsep{.2in} \xymatrixcolsep{.2in} \xymatrix{&\Br(L)&\\\Br(k_1)\ar^{\Res}[ru]\ar@{<->}[rr]^{\Phi_{\Br}}&&\Br(k_2).\ar[lu]_{\Res}\\&\Br(F)\ar[ru]_\Res \ar[lu]^\Res&}
\end{equation}

In particular this shows that we may obtain central simple algebras over $k_1$ and $k_2$ having the same isomorphism classes of subalgebras defined over $F$ by considering the pair $[A],\Phi_{\Br}([A])$ for any $[A]\in\Br(k_1)$. Call two number fields $k_1$ and $k_2$ {\bf Brauer equivalent} if there is an isomorphism $\Phi_{\Br}: \Br(k_1)\rightarrow \Br(k_2)$ for which the diagram in (\ref{eqnBrauerNatural}) commutes for all $F\subset k_1\cap k_2$ and $L\supset k_1,k_2$. 

This paper studies Brauer equivalent number fields with an eye towards geometric applications. Before describing this paper's results, we recall that the Galois kernel of a number field $k$ is the largest subfield of $k$ which is Galois over $\mathbb Q$. Algebraically, the main result of the paper is:

\begin{thm}\label{thm:main}
Let $k$ and $k'$ be Brauer equivalent number fields. The following invariants of $k$ and $k'$ coincide:
\begin{enumerate}
\item Their degrees.
\item Their signatures.
\item Their Galois kernels.
\item Their groups of roots of unity.
\end{enumerate}
\end{thm}

As an immediate consequence of Theorem \ref{thm:main} we obtain:
\begin{cor}
If $k$ and $k'$ are Brauer equivalent number fields and $k/\mathbb Q$ is Galois then $k\cong k'$.
\end{cor}

These results have a number of geometric consequences for locally symmetric spaces arising from quaternion algebras. In Corollary \ref{cor2}, for instance, we prove that if $M_1$ and $M_2$ are incommensurable arithmetic hyperbolic $3$-manifolds with Brauer equivalent fields of definition $k_1$ and $k_2$  then $M_1$ and $M_2$ have no proper immersed totally geodesic surfaces in common. On the other hand, in Theorem \ref{locallysymmetricconstruction} we construct incommensurable arithmetic locally symmetric spaces containing exactly the same set of proper immersed totally geodesic surfaces. This provides a striking contrast to the behavior of arithmetic hyperbolic $3$-manifolds and refines the work of McReynolds \cite{McR}, where the manifolds constructed were only shown to have the same commensurability classes of totally geodesic submanifolds.

\begin{ack}The author thanks Thomas Shemanske for helpful conversation regarding the proof of Theorem \ref{maxorders}. The author was partially supported by the U.S. National Science Foundation grants DMS 1107452, 1107263, 1107367 ``RNMS: Geometric Structures and Representation Varieties'' (the GEAR Network) and a Simons Collaboration Grant.
\end{ack}

\section{Preliminaries and notation}

In this section we will describe the notation that will be used throughout this paper.

Let $k$ be a number field with ring of integers $\mathcal O_k$. We will denote the number of real (respectively complex) places of $k$ by $r_1$ (respectively $r_2$). The signature of $k$ is the pair $(r_1,r_2)$. Note that if $n=[k:\mathbb Q]$ then $n=r_1+2r_2$. We denote by $\mathscr P_k$ the set of all primes (finite or infinite) of $k$. We will denote the Galois closure of $k$ by $\widehat{k}$ and the Galois kernel of $k$ by $k_0$. We remind the reader that the Galois kernel of a number field $k/\mathbb Q$ is the largest subfield of $k$ which is Galois over $\mathbb Q$. 

Given an extension $L/k$ of number fields and a prime $\mathfrak P\in\mathscr P_L$ there is a unique prime $\mathfrak p\in\mathscr P_k$ lying beneath $\mathfrak P$ (i.e., $\mathfrak P \cap \mathcal O_k = \mathfrak p$). In this situation we write $\mathfrak P\mid \mathfrak p$. 

We will denote the Brauer group of $k$ by $\Br(k)$, and if $L$ is a finite degree extension of $k$ then we will denote the relative Brauer group of $L/k$ by $\Br(L/k)$. This is defined to be the kernel of the restriction map $\Res_{L/k} : \Br(k) \rightarrow \Br(L).$ 

We will typically use $A$ or $B$ to denote central simple algebras over $k$ and $D$ to denote a central division algebra over $k$. Given a central simple algebra $A$ over $k$ there is a unique central division algebra $D$ over $k$ and nonnegative integer $r$ such that $A\cong \MM_r(D)$. Therefore \[\dim_k(A)=r^2\cdot \dim_k(D).\] We define the {\bf degree} of $A$ to be $\sqrt{\dim_k(A)}$ and denote it by $[A:k]$. We define the (Schur) {\bf index} of $A$, denoted $\ind(A)$, to be $[D:k]$. That is, the index of $A$ is the degree of its underlying central division algebra.

Given a central simple algebra $A$, we will denote by $A^\times$ the multiplicative group of invertible elements of $A$ and by $A^1$ the multiplicative group of elements of $A$ with reduced norm $1$. Let $A$ be a central simple algebra over $k$ of dimension $n^2$. Given a prime $\mathfrak p\in\mathscr P_k$ we have an associated central simple algebra $A_\mathfrak p = A\otimes_k k_\mathfrak p$ over $k_\mathfrak p$. If $A_\mathfrak p\cong \MM_n(k_\mathfrak p)$ then we say that the prime $\mathfrak p$ is {\bf unramified} (or split) in $A$. Otherwise $A_\mathfrak p \cong \MM_m(D_\mathfrak p)$ for some positive integer $m<n$ and central division algebra $D_\mathfrak p$ over $k_\mathfrak p$. In this case we say that $\mathfrak p$ is {\bf ramified} in $A$. We will denote by $\Ram(A)$ the set of primes of $k$ which ramify in $A$, by $\Ram_f(A)$ the set of finite primes of $k$ ramifying in $A$ and by $\Ram_\infty(A)$ the set of infinite primes of $k$ that ramify in $A$. 

We now discuss the local invariants of $A$. Let $\mathfrak p\in \mathscr P_k$. We begin with the case that $\mathfrak p$ is an infinite prime of $k$ so that $k_\mathfrak p=\mathbb R$ or $k_\mathfrak p=\mathbb C$. In this case we define the local invariant of $A$ at $\mathfrak p$, denoted $\inv_\mathfrak p (A)$, by $\inv_\mathfrak p(A)=0$ when $k_\mathfrak p = \mathbb C$,  $\inv_\mathfrak p(A)=0$ when $k_\mathfrak p = \mathbb R$ and $A_\mathfrak p \cong \MM_n(\mathbb R)$ and $\inv_\mathfrak p(A)=\frac{1}{2}$ when $k_\mathfrak p = \mathbb R$ and $A_\mathfrak p \cong \MM_{\frac{n}{2}}(\mathbb H)$. Now suppose that $\mathfrak p$ is a finite prime of $k$. In this case there is an isomorphism $\inv: \Br(k_\mathfrak p)\rightarrow \mathbb Q/\mathbb Z$ (see \cite[Theorem 31.8]{Reiner}). We define the invariant of $A$ at $\mathfrak p$ by $\inv_\mathfrak p(A)=\inv([A_\mathfrak p])$. Finally, if $L_\mathfrak P/k_\mathfrak p$ is a finite extension and $A_\mathfrak p$ is a central simple algebra over $k_\mathfrak p$ then by \cite[Theorem 31.9]{Reiner} we have \[\inv_\mathfrak P(A_\mathfrak p \otimes_{k_\mathfrak p} L_\mathfrak P)=[L_\mathfrak P: k_\mathfrak p]\cdot \inv_\mathfrak p(A_\mathfrak p).\] 

\subsection{Quaternionic Shimura varieties}

In this subsection we review the construction of arithmetic lattices acting on products of hyperbolic planes and $3$-dimensional hyperbolic spaces. For a complete treatment we refer the reader to \cite[Section 3]{B}.

Let $k$ be a number field with ring of integers $\mathcal O_k$ and signature $(r_1,r_2)$ and recall that $[k:\mathbb Q]=r_1+2r_2$. If $A$ is a quaternion algebra over $k$ which is not totally definite (i.e., there exists an infinite prime of $k$ which splits in $A$) then there exist non-negative integers $r,s$ with $r+s=r_1$ such that \[A\otimes_{\mathbb Q} \mathbb R \cong \MM_2(\mathbb R)^s \times \mathbb H^r \times \MM_2(\mathbb C)^{r_2}.\] This isomorphism induces an embedding \[\pi: A^\times \hookrightarrow \prod_{\nu\not\in \Ram_{\infty}(A)} (A\otimes_k k_\nu)^\times \longrightarrow \GL_2(\mathbb R)^s \times \GL_2(\mathbb C)^{r_2}.\] Restricting to the subgroup $A^1$ of elements of $A$ with reduced norm $1$ yields an embedding \[\pi: A^1 \hookrightarrow \SL_2(\mathbb R)^{s} \times \SL_2(\mathbb C)^{r_2}.\]

If $\mathcal O$ is a maximal order of $A$ and $\mathcal O^1$ the multiplicative group of those elements of $\mathcal O$ with reduced norm $1$ then the image $\pi(\mathcal O^1)$ of $\mathcal O^1$ in $\SL_2(\mathbb R)^{s} \times \SL_2(\mathbb C)^{r_2}$ is a  lattice by the work of Borel and Harish-Chandra \cite{BHC}. We will denote this lattice by $\Gamma_\mathcal O$. We define an irreducible lattice in $\SL_2(\mathbb R)^{s} \times \SL_2(\mathbb C)^{r_2}$ to be {\bf arithmetic} if it is commensurable with a lattice of the form $\Gamma_{\mathcal O}$. Note that the arithmetic lattice will be cocompact if and only if its associated quaternion algebra is not isomorphic to the matrix algebra $\MM_2(k)$.

Let $G=\SL_2(\mathbb R)^{s} \times \SL_2(\mathbb C)^{r_2}$ and $K$ be a maximal compact subgroup of $G$. If $\Gamma$ is an arithmetic lattice of $G$ then the quotient space $M=\Gamma\backslash G/K$ is an arithmetic locally symmetric space called a {\bf quaternionic Shimura variety}. We call $k$ the {\bf invariant trace field} of $M$ and $A$ the {\bf invariant quaternion algebra} of $M$.

If $M_1$ and $M_2$ are quaternionic Shimura varieties then we say that $M_1$ and $M_2$ are {\bf commensurable} if they share an isometric finite-sheeted covering. It is well-known that $M_1$ and $M_2$ are commensurable if and only if they have isomorphic invariant trace fields and invariant quaternion algebras (c.f. \cite[Theorem 8.4.1]{MR}).

There are two special cases of the above construction that are worth highlighting.  If $k$ is a totally real number field for which $G=\SL_2(\mathbb R)$ (i.e., $r=r_1-1$ and $s=1$) then we call any lattice commensurable with $\Gamma_{\mathcal O}$ an {\bf arithmetic Fuchsian group}. If $\Gamma_{\mathcal O}$ is torsion-free then the resulting quotient space $M$ is an {\bf arithmetic hyperbolic surface}. If $k$ is a number field containing a unique complex place and $G=\SL_2(\mathbb C)$ (i.e., $r=r_1$ and $s=0$) then we call any lattice commensurable with $\Gamma_{\mathcal O}$ an {\bf arithmetic Kleinian group} and $M$ an {\bf arithmetic hyperbolic $3$-manifold} when $\Gamma_{\mathcal O}$ is torsion-free.

\section{Proof of Theorem \ref{thm:main}}

We will begin by proving that if $k$ and $k'$ are Brauer equivalent then $[k:\mathbb Q]=[k':\mathbb Q]$. Let $L$ be a number field containing $k$ and $k'$ such that $L/k$ and $L/k'$ are both Galois extensions. Because $k$ and $k'$ are Brauer equivalent there exists an isomorphism $\Phi_{\mathrm{Br}}: \mathrm{Br}(k)\rightarrow \mathrm{Br}(k')$ for which the following diagram commutes:

\begin{equation}\label{BrauerNatural} 
\xymatrixrowsep{.2in} \xymatrixcolsep{.2in} \xymatrix{&\Br(L)&\\\Br(k)\ar^{\Res_L}[ru]\ar@{->}[rr]^{\Phi_{\Br}}&&\Br(k')\ar[lu]_{\Res_L}\\}
\end{equation}

In particular this implies that $\Phi_{\Br}(\Br(L/k))=\Br(L/k')$. Let $D$ be a division algebra over $k$ of degree $[D:k]=[L:k]$ such that $[D]\in\Br(L/k)$. By \cite[Chapter 13.5]{P}, there exists a central simple algebra $S$ over $k'$ with $[S:k']=[L:k']$ and $S\in\Phi_{\Br}([D])$. Observe that 
\begin{align*}
[L:k] &= [D:k]\\
&= \ind([D]) \\
&= \exp([D]) \\
&= \exp(\Phi_{\Br}([D])) \\
&= \ind(\Phi_{\Br}([D])) \\
&\leq [S:k] \\ 
&= [L:k'].
\end{align*}

Here we have used the fact that if $K$ is a number field and $[A]\in\Br(K)$ then $\ind([A])=\exp([A])$ \cite[Theorem 32.19]{Reiner}, where $\ind([D])$ is the Schur index of $D$ and $\exp([D])$ is the order of $[D]$ in $\Br(K)$. This shows that $[L:k]\leq [L:k']$. Interchanging the roles of $k$ and $k'$ proves that $[L:k]=[L:k']$. We now see that \[[k:\mathbb Q]=[L:\mathbb Q]/[L:k]=[L:\mathbb Q]/[L:k']=[k':\mathbb Q].\]

That $k$ and $k'$ have the same degrees and isomorphic Brauer groups imply that their signatures are the same, since the exact sequence of Brauer groups in class field theory \cite[(32.13)]{Reiner} implies that if $\Br(k)\cong \Br(k')$ then $k$ and $k'$ have the same number of real places.

We now show that $k$ and $k'$ have the same Galois kernels. Let $k_0$ denote the Galois kernel of $k$ and $k'_0$ denote the Galois kernel of $k'$. Let $\widehat{k}$ denote the Galois closure over $\mathbb Q$ of $k$. We claim that $k'_0\subset \widehat{k}$. To see this we will use a fact proven in \cite[Theorem 1.4]{LMM}: If $k$ and $k'$ are Brauer equivalent number fields then for every prime $p$ which is unramified in both $k$ and $k'$ we have 
\[\gcd(f(v_1/p),\dots,f(v_g/p))=\gcd(f(v'_1/p),\dots,f(v'_{g'}/p)),\]
where $v_1,\dots,v_g$ are the places of $K$ lying above $p$ and $v'_1,\dots,v'_{g'}$ are the places of $k'$ lying above $p$. Here $f(v_i/p)$ (respectively $f(v'_i/p)$) is the inertia degree of $v_i$ (respectively $v'_i$) over $p$. Suppose that $p$ is a prime that is unramified in $k$ and $k'$ and splits completely in $\widehat{k}/\mathbb Q$. Then $p$ also splits completely in $k/\mathbb Q$, hence $\gcd(f(v_1/p),\dots,f(v_g/p))=\gcd(f(v'_1/p),\dots,f(v'_{g'}/p))=1$. Let $v'_0$ be a prime of $k'_0$ lying above $p$. Then $f(v'_0/p)$ divides $f(v'/p)$ for all primes $v'$ of $k'$ lying above $p$, hence $f(v'_0/p)=1$. Because $k'_0/\mathbb Q$ is Galois this implies that $p$ splits completely in $k'_0/\mathbb Q$. It follows from \cite[Theorem 9, p. 168]{Lang-ANT} that $k'_0$ is contained in $\widehat{k}$, proving the claim.

Brauer's theorem \cite[p. 135]{CFT} now implies that if every prime $p$ having a degree one factor in $k$ splits completely in $k_0'/\mathbb Q$ (up to a set of primes with density $0$) then $k'_0\subset k$. In our situation having a degree one factor in $k$ forces \[\gcd(f(v_1/p),\dots,f(v_g/p))=1=\gcd(f(v'_1/p),\dots,f(v'_{g'}/p)),\] and consequently implies that $p$ must split completely in $k'_0/\mathbb Q$. This shows that $k_0'\subset k$. Since $k_0$ and $k'_0$ are both Galois extensions of $\mathbb Q$ contained in $k$, their compositum must be a Galois extension of $\mathbb Q$ contained in $k$. Since $k_0$ is by definition the largest Galois extension of $\mathbb Q$ contained in $k$, this implies that the compositum of $k_0$ and $k'_0$ is equal to $k_0$, hence $k'_0\subset k_0$. By symmetry $k_0\subset k'_0$ and $k_0=k'_0$.

That $k$ and $k'$ have the same groups of roots of unity follows from their having the same Galois kernels because the roots of unity contained in each field generates a Galois extension.

\section{Totally geodesic surfaces of arithmetic hyperbolic $3$-manifolds}

In this section we prove that incommensurable arithmetic hyperbolic $3$-manifolds defined over Brauer equivalent number fields never have any proper immersed totally geodesic surfaces in common. Our proof will make use of the following algebraic result.

\begin{thm}\label{tensorupthm}
Suppose that $F$ is a number field and $B_0/F$ is a quaternion algebra. Let $k, k'/F$ be Galois extensions of degree $2^n$ and define $A_1=B_0\otimes_F k$ and $A_2=B_0\otimes_F k'$. If  $B\otimes_F k\cong A_1$ if and only if $B\otimes_F k'\cong A_2$ for every quaternion algebra $B/F$ with $\Ram_\infty(B)=\Ram_\infty(B_0)$ then $k\cong k'$ and $A_1\cong A_2$.
\end{thm}
\begin{proof}
Let $\nu_1,\nu_2$ be primes of $F$ which do not split completely in $k/F$ and do not ramify in $B_0$. We will show that $\nu_1$ and $\nu_2$ do not split completely in $k'/F$ either. It follows that the set of primes of $F$ which split completely in $k'/F$ is (with at most a finite number of exceptions) a subset of the set of primes of $F$ which split completely in $k/F$. A standard consequence of the Chebotarev density theorem (cf. \cite[Thm 9, p.~168]{Lang-ANT}) then shows that $k$ is isomorphic to a subfield of $k'$, hence the two fields are themselves isomorphic since their degrees over $F$ are equal.

Define a quaternion algebra $B$ over $F$ by setting $\Ram_\infty(B)=\Ram_\infty(B_0)$ and $\Ram_f(B)=\Ram_f(B_0)\cup \{\nu_1, \nu_2\}$. We claim that if $\omega$ is a prime of $k'$ lying above $\nu_1$ then $\inv_\omega(A_2)=0$ in $\mathbb Q/\mathbb Z$. Indeed, this follows from

\begin{align}\label{invs}
\begin{split}
\inv_\omega(A_2) &= \inv_\omega(B_0\otimes_F k') \\
&= \inv_{\nu_1}(B_0)\cdot [k'_\omega : F_{\nu_1}] \\
&= [k'_\omega : F_{\nu_1}],
\end{split}\end{align}

where the last equality follows from the fact that $\nu_1\not\in\Ram(B_0)$, hence $\inv_{\nu_1}(B_0)$ is trivial in $\mathbb Q/\mathbb Z$. Since $[k'_\omega : F_{\nu_1}]$ is an integer, the claim follows.

We now show that $B\otimes_F k\cong A_1$. It suffices to show that $\Ram_f(B\otimes_F k)=\Ram_f(A_1)$. By hypothesis $A_1\cong B_0\otimes_F k$, hence we may show that $\Ram_f(B\otimes_F k)=\Ram_f(B_0\otimes_F k)$. Suppose first that $\frakP$ is a prime of $k$ lying in $\Ram_f(B_0\otimes_F k)$ and let $\frakp$ be the prime of $F$ lying beneath $\frakP$. Then \[\frac{1}{2}=\inv_\frakP(B_0\otimes_F k)=\inv_\frakp(B_0)\cdot [k_\frakP:F_\frakp].\] By Lemma 3.2 of \cite{LS}, the local degree $ [k_\frakP:F_\frakp]$ is odd, hence $\inv_\frakp(B_0)=\frac{1}{2}$ and $\frakp\in\Ram(B_0)$. Since $\Ram(B)$ contains $\Ram(B_0)$, $\frakp$ ramifies in $B$. Then \[\inv_\frakP(B\otimes_F k)=\inv_\frakp(B)\cdot [k_\frakP:F_\frakp]=\frac{1}{2}\cdot [k_\frakP:F_\frakp]\] is nontrivial in $\mathbb Q/\mathbb Z$. This shows that $\Ram_f(B_0\otimes_F k)\subseteq \Ram_f(B\otimes_F k)$. 

We now show the reverse containment. Suppose that $\frakQ\in\Ram_f(B\otimes_F k)$ and let $\frakq$ be the prime of $F$ lying beneath $\frakQ$. The argument above shows that $\frakq\in\Ram(B)$. If $\frakq\in\Ram(B_0)$ then it is necessarily the case that $\frakQ\in\Ram_f(B_0\otimes_F k)$. Assume therefore that $\frakq\not\in \Ram(B_0)$. Because $\Ram_f(B)=\Ram_f(B_0)\cup \{\nu_1, \nu_2\}$ it must be the case that $\frakq\in\{\nu_1,\nu_2\}$. We now obtain a contradiction as \[\inv_\frakQ(B\otimes_F k)=\inv_\frakp(B)\cdot [k_\frakQ:F_\frakp]\in \mathbb Z\] since $k/F$ is a Galois extension of degree $2^n$, hence the local degree is even if it is not $1$ (in which case it corresponds to a prime of $F$ which splits completely in $k/F$, which contradicts the way we chose $\nu_1$ and $\nu_2$). This shows that $A_1\cong B\otimes_F k$.

By hypothesis we have $A_2\cong B\otimes_F k'$. Let $\omega$ be a prime of $k'$ lying above $\nu_1$. By (\ref{invs}), the local invariant $\inv_\omega(A_2)$ is trivial in $\mathbb Q/\mathbb Z$, giving us \[\inv_\omega(A_2)=\inv_\frakp(B)\cdot [k'_\omega:F_{\nu_1}]=\frac{1}{2}\cdot [k'_\omega:F_{\nu_1}]\in \mathbb Z.\] This forces the local degree $[k_\omega:F_{\nu_1}]$ to be even and allows us to conclude that $\nu_1$ does not split completely in $k'/F$. As the argument that $\nu_2$ does not split completely in $k'/F$ is identical, we have completed the proof of the theorem.
\end{proof}

\begin{cor}\label{cor1}If $k$ and $k'$ are Brauer equivalent number fields then there is no subfield $F\subset k\cap k'$ over which $k$ and $k'$ are Galois with relative degree a power of $2$.\end{cor}
\begin{proof}
Let $F$ be a subfield of $k\cap k'$ such that the relative extensions $k/F$ and $k'/F$ are Galois of degree $2^n$. Let $B_0$ be a quaternion division algebra over $F$ and define $A_1=B_0\otimes_F k$ and $A_2=B_0\otimes_F k'$. 

Commutativity of the following diagram implies that $\Phi_{\Br}([A_1])=[A_2]$: 
\begin{equation*}
\xymatrix{&\\\Br(k)\ar@{<->}[rr]^{\Phi_{\Br}}&&\Br(k')\\&\Br(F)\ar[ru]_{\Res_{k'}} \ar[lu]^{\Res_k}&}
\end{equation*}

In particular this diagram and the uniqueness of the division algebra representative of a Brauer class shows that if $B$ is a quaternion algebra over $F$ then $A_1\cong B\otimes_F k$ if and only if $A_2\cong B\otimes_F k'$. The corollary now follows from Theorem \ref{tensorupthm}.
\end{proof}

\begin{cor}\label{cor2}
If $M_1$ and $M_2$ are incommensurable arithmetic hyperbolic $3$-manifolds with Brauer equivalent invariant trace fields $k_1$ and $k_2$  then $M_1$ and $M_2$ have no proper immersed totally geodesic surfaces in common.
\end{cor}
\begin{proof}
Suppose that $M_1$ and $M_2$ are as in the statement of the corollary and that they both contain an immersed totally geodesic surface $S$. It follows from \cite[Chapter 9.5]{MR} that $S$ is arithmetic and that the invariant trace field of $S$ if a totally real number field $F$ satisfying $[k:F]=2=[k':F]$. Therefore $F=k\cap k'$ and $k/F, k'/F$ are quadratic extensions, hence Galois. Corollary \ref{cor2} now follows from Corollary \ref{cor1}.
\end{proof}

\section{Maximal orders}

In this section we prove a variety of algebraic results concerning maximal orders in quaternion algebras defined over Brauer equivalent number fields. These results will be applied in Section \ref{section:locallysymmetric} when we construct incommensurable arithmetic locally symmetric spaces containing exactly the same set of proper immersed totally geodesic surfaces.

\begin{prop}\label{samealgebras}
Let $K_1$ and $K_2$ be Brauer equivalent number fields and $F\subset K_1\cap K_2$ a common subfield. Let $B_0$ be a quaternion algebra over $F$ and define $A_1=B_0\otimes_F K_1$ and $A_2=B_0\otimes_F K_2$. Assume that $A_1$ and $A_2$ are division algebras. If $B$ is a quaternion algebra over $F$ then $A_1\cong B\otimes_F K_1$ if and only if $A_2\cong B\otimes_F K_2$.
\end{prop}
\begin{proof}
Suppose that $B$ is a quaternion algebra over $F$ such that $A_1\cong B\otimes_F K_1$. We will show that $A_2\cong B\otimes_F K_2$. This suffices to prove the theorem as the same argument, with the roles of $A_1$ and $A_2$ reversed will prove the other direction. Because $K_1$ and $K_2$ are Brauer equivalent there is an isomorphism $\Phi_{\Br}: \Br(K_1)\rightarrow \Br(K_2)$ which commutes with extension of scalars. It follows that \[\Phi_{\Br}([A_1])=\Phi_{\Br}([B_0\otimes_F K_1])=[B_0\otimes_F K_2]=[A_2].\] The same reasoning implies that \[[A_2]=\Phi_{\Br}([A_1])=\Phi_{\Br}([B\otimes_F K_1])=[B\otimes_F K_2].\] Since $A_2$ and $B\otimes_F K_2$ are both quaternion algebras over $K_2$ ,with the former a division algebra, and every Brauer class contains a unique division algebra representative, we conclude that $A_2\cong B\otimes_F K_2$, as desired.
\end{proof}

\begin{prop}\label{typenumber}
Let $K$ be a number field, $A$ be a quaternion algebra over $K$ which is unramified at some infinite prime of $K$ and $\mathcal O$ be a maximal order in $A$. If $K$ has narrow class number one then every maximal order in $A$ is conjugate to $\mathcal O$.
\end{prop}
\begin{proof}
Section 3 of \cite{L} shows that the number of conjugacy classes of maximal orders in $A$ is equal to the degree over $K$ of a certain class field $K(A)$. This class field is defined as the maximal abelian extension of $K$ which has a $2$-elementary Galois group, is unramified outside of the real places of $K$ which ramify in $A$, and in which all finite primes of $K$ that ramify in $A$ split completely. It follows that $K(A)$ is contained in the narrow class field of $K$, which by hypothesis is equal to $K$. The theorem follows.
\end{proof}

\begin{thm}\label{maxorders}
Let $K/F$ be an extension of number fields, let $A$ be a quaternion algebra over $K$, $B$ be a quaternion algebra over $F$ and suppose that $A\cong B\otimes_F K$. Identify $B$ with its image in $A$. If $\mathcal O$ is a maximal order of $A$ then $\mathcal O_B=\mathcal O\cap B$ is a maximal order of $B$ and $\mathcal O_B\otimes_{\mathcal O_F} \mathcal O_K \cong \mathcal O$.
\end{thm}
\begin{proof}
We begin by showing that $(\mathcal O\cap B)\otimes_{\mathcal O_F} \mathcal O_K\cong \mathcal O$. As $\mathcal O$ and $B$ are $\mathcal O_F$-modules and $\mathcal O_K$ is flat over $\mathcal O_F$ we have that tensor products with $\mathcal O_K$ commute with (finite) intersections,  yielding
\begin{align*}
(\mathcal O\cap B)\otimes_{\mathcal O_F} \mathcal O_K &\cong (\mathcal O\otimes_{\mathcal O_F} \mathcal O_K)\cap (B\otimes_{\mathcal O_F}\mathcal O_K) \\
&\cong \mathcal O \cap (B\otimes_{F}(F \otimes_{\mathcal O_F} \mathcal O_K))\\
&\cong \mathcal O \cap (B\otimes_F K)\\
&\cong \mathcal O \cap A \\
&= \mathcal O.
\end{align*}

We now show that  the order $\mathcal O\cap B$ is a maximal order of $B$. Let $\Lambda$ be a maximal order of $B$ containing $\mathcal O\cap B$. In light of the previous paragraph we see that $\Lambda\otimes_{\mathcal O_F} \mathcal O_K \cong \mathcal O$. In order to show that $\mathcal O \cap B=\Lambda$ it suffices to show that the discriminants $d(\mathcal O \cap B)$ and $d(\Lambda)$ of these orders are equal (by \cite[exer. 3, p. 131]{Reiner}). Exercise 6 of \cite[p. 131]{Reiner} shows that because $(\mathcal O\cap B)\otimes_{\mathcal O_F}\mathcal O_K$ and $\Lambda\otimes_{\mathcal O_F}\mathcal O_K$ are isomorphic, we have that $d(\mathcal O \cap B)\mathcal O_K=d(\Lambda)\mathcal O_K$. This shows that the discriminants of $\mathcal O \cap B$ and $\Lambda$ (these discriminants are ideals of $\mathcal O_F$) extend to the same ideal of $\mathcal O_K$. But different ideals of $\mathcal O_F$ must extend to different ideals of $\mathcal O_K$, hence $d(\mathcal O \cap B)=d(\Lambda)$ and $\Lambda=\mathcal O\cap B$ is maximal.\end{proof}

\section{Shimura varieties containing the same immersed totally geodesic surfaces}\label{section:locallysymmetric}

Let $K_1=(\sqrt[8]{-6561})$ and $K_2=(\sqrt[8]{-16\cdot 6561})$. It was shown by de Smit and Perlis \cite[p. 214]{DP} that $K_1$ and $K_2$ are non-isomorphic number fields with isomorphic adele rings, hence by Theorem 1.1 of \cite{LMM} they are Brauer equivalent. We also note that the only subfield that $K_1$ and $K_2$ have in common is $\mathbb Q$. Below we will also make use of the observation that $\mathbb Q$ is the only totally real subfield of $K_1$ and of $K_2$.  Finally, using the computer algebra system SAGE \cite{sage} one can check that both $K_1$ and $K_2$ have narrow class number one. As a consequence, Theorem \ref{typenumber} implies that in every quaternion algebra over $K_1$ and $K_2$, all maximal orders are conjugate.

Let $A_1=\MM_2(K_1)$ and $A_2=\MM_2(K_2)$. Denote by $\mathcal O_{K_1}$ and $\mathcal O_{K_2}$ the rings of integers of $K_1$ and $K_2$ and define $\mathcal O_1=\MM_2(\mathcal O_{K_1})$ and $\mathcal O_2=\MM_2(\mathcal O_{K_2})$. These are maximal orders in $A_1$ and $A_2$.  Because $K_1$ and $K_2$ have signature $(0,4)$ there are isomorphisms \[A_1\otimes_{\mathbb Q} \mathbb R \cong \MM_2(\mathbb C)^4 \cong A_2\otimes_{\mathbb Q} \mathbb R.\] Denote by $\Gamma_1$ and $\Gamma_2$ the images in $\SL_2(\mathbb C)^4$ of $\SL_2(\mathcal O_{K_1})$ and $\SL_2(\mathcal O_{K_2})$. The lattices $\Gamma_1$ and $\Gamma_2$ are arithmetic, non-cocompact and have finite covolume.

Let $G=\SL_2(\mathbb C)^4$, $K$ be a maximal compact subgroup of $G$ and define $M_1=\Gamma_1\backslash G/K$ and  $M_2=\Gamma_2\backslash G/K$. The orbifolds $M_1$ and $M_2$ are (incommensurable) arithmetic locally symmetric spaces.

\begin{thm}\label{locallysymmetricconstruction}
The arithmetic locally symmetric spaces $M_1$ and $M_2$ contain exactly the same set of proper immersed totally geodesic surfaces.
\end{thm}
\begin{proof}
Suppose that $S$ is a totally geodesic surface in $M_1$ and denote by $\Gamma$ the orbifold fundamental group of $S$. Then $\Gamma\subset \Gamma_1$. We will show that $\Gamma\subset \Gamma_2$. Let $B$ denote the invariant quaternion algebra of $\Gamma$ and $F$ denote invariant quaternion algebra of $\Gamma$. The field $F$ is a totally real subfield of $K_1$, hence by the observation we made above $F=\mathbb Q$. As it is necessarily the case that $B\otimes_F K_1\cong A_1$, we conclude by Theorem \ref{maxorders} the intersection $\mathcal O_B=\mathcal O_1 \cap B$ is a maximal order of $B$ for which $\mathcal O_B\otimes_{\mathbb Z} \mathcal O_{K_1} \cong \mathcal O_1$. It follows that $\Gamma \subset \Gamma_{\mathcal O_B} \subset \Gamma_1$. 

We now show that (up to isomorphism) $\Gamma$ is contained in $\Gamma_2$. Because $K_1$ and $K_2$ are Brauer equivalent and $[B\otimes_F K_1]=[A_1]=[\MM_2(K_1)]$ is trivial in $\Br(K_1)$, it must be the case that $[B\otimes_F K_2]$ is trivial in $\Br(K_2)$. In particular this shows that $B\otimes_F K_2\cong \MM_2(K_2)=A_2$. Identifying $B$ with its image in $A_2$ we have that the maximal order $\mathcal O_B$ is contained in the order $\mathcal O_B\otimes_{\mathbb Z} \mathcal O_{K_2}$ of $A_2$. Because every maximal order of $K_2$ is conjugate (by Theorem \ref{typenumber}), we conclude that (perhaps upon conjugating $\mathcal O_B$) we have $\mathcal O_B \subset \mathcal O_2$, hence $\Gamma_{\mathcal O_B} \subset \Gamma_2$. As $\Gamma \subset \Gamma_{\mathcal O_B}$, this proves that $\Gamma\subset \Gamma_2$. This shows that every proper immersed totally geodesic surface of $M_1$ is also a proper immersed totally geodesic surface of $M_2$. Repeating the same argument with the roles of $M_1$ and $M_2$ reversed completes the proof of the theorem.
\end{proof}

\begin{thm}
Let $K_1$ and $K_2$ be non-isomorphic Brauer equivalent number fields which are primitive (i.e., have no proper subfields other than $\mathbb Q$). There are infinitely many pairs of incommensurable arithmetic locally symmetric spaces defined over $K_1$ and $K_2$ that contain exactly the same set of proper immersed totally geodesic surfaces. 
\end{thm}
\begin{proof}
We begin by proving that there are infinitely many quaternion division algebras $A_1$ defined over $K_1$ satisfying:
\begin{enumerate}[(1)]
\item $A_1$ is indefinite (i.e., some infinite prime of $K_1$ is unramified in $A_1$),
\item all maximal orders of $A_1$ are conjugate, and
\item there exists a quaternion algebra $B$ over $\mathbb Q$ such that $B\otimes_{\mathbb Q} K_1\cong A_1$.
\end{enumerate}

Let $p_1,p_2$ be rational primes which split completely in $K_1/\mathbb Q$ and define a quaternion division algebra $B$ over $\mathbb Q$ by $\Ram(B)=\{p_1,p_2\}$. Let $A=B\otimes_{\mathbb Q} K_1$. If $\omega$ is a prime of $K_1$ lying above one of the $p_i$ then $[(K_1)_\omega : \mathbb Q_{p_i}]=1$ (as $p_1$ splits completely in $K_1/\mathbb Q$), hence \[\inv_\omega(A)=\inv_\omega(B\otimes_{\mathbb Q} K_1)=[(K_1)_\omega : \mathbb Q_{p_i}]\cdot \inv_{p_i}(B)=\frac{1}{2}\in\mathbb Q/\mathbb Z,\] which shows that $\omega\in\Ram(A)$. In particular this shows that $A$ is a division algebra.

As was mentioned in the proof of Proposition \ref{typenumber}, if $A$ is a quaternion algebra over a number field $K$ which is unramified at an infinite prime of $K$ then Section 3 of \cite{L} implies that the number of conjugacy classes of maximal orders in $A$ is equal to the degree over $K$ of the maximal abelian extension $K(A)$ of $K$ which has a $2$-elementary Galois group, is unramified outside of the real places of $K$ which ramify in $A$, and in which all finite primes of $K$ that ramify in $A$ split completely. This shows that if $K_1(A)=K_1$ then the algebra $A$ satisfies conditions (1)-(3) above.

Suppose therefore that $K_1(A)\neq K_1$ and that $A$ does not satisfy (2) above. Let $p_3,p_4$ be rational primes (distinct from $p_1, p_2$) which split completely in $K_1/\mathbb Q$ but not in $K_1(A)/\mathbb Q$ and define a quaternion algebra $B_1$ over $\mathbb Q$ by $\Ram(B_1)=\{p_1,p_2,p_3,p_4\}$. Let $A_1=B_1\otimes_\mathbb Q K_1$. The definition of $K_1(A)$ and $K_1(A_1)$ implies that $K_1\subseteq K_1(A_1)\subsetneq K_1(A)$. If $K_1(A_1)=K_1$ then $A_1$ is the desired quaternion algebra. Otherwise we may continue this construction to obtain a quaternion algebra $A_2$ with \[K_1\subseteq K_1(A_2)\subsetneq K_1(A_1)\subsetneq K_1(A).\] Because the degree of $K_1(A)$ over $K_1$ is finite, we may continue this process some finite number of times so as to obtain a quaternion division algebra over $K_1$ satisfying conditions (1)-(3) above. 

The construction above shows that there exists a quaternion algebra $A_1'$ over $K_1$ satisfying conditions (1)-(3) above. Similarly, we may construct a quaternion algebra $A_2'$ over $K_2$ satisfying properties (1)-(3) as well (with $K_2$ in place of $K_1$). If $A_1'\cong B_1\otimes_{\mathbb Q} K_1$ and $A_2'\cong B_2\otimes_{\mathbb Q} K_2$ then define a quaternion algebra $B$ over $\mathbb Q$ by $\Ram(B)=\Ram(B_1)\cup \Ram(B_2)$. The class field containments used above show that if we define $A_1:=B\otimes_\mathbb Q K_1$ and $A_2:=B\otimes_\mathbb Q K_2$ then $K_1(A_1)=K_1$ and $K_2(A_2)=K_2$. In particular all of the maximal orders in $A_1$ (respectively $A_2$) are conjugate.

Let $(a,b)$ be the signature of $K_1$ (and hence of $K_2$ by Theorem \ref{thm:main}) so that \[A_1\otimes_\mathbb Q\mathbb R \cong \MM_2(\mathbb R)^a \times \MM_2(\mathbb C)^b\cong A_2\otimes_\mathbb Q\mathbb R .\] Let $G=\SL_2(\mathbb R)^a\times \SL_2(\mathbb C)^b$ and $K$ be a maximal compact subgroup of $G$. Let $\mathcal O_1$ be a maximal order of $A_1$ and $\mathcal O_2$ be a maximal order of $A_2$. Finally, define $M_1=\Gamma_{\mathcal O_1}\backslash G/K$ and $M_2=\Gamma_{\mathcal O_2}\backslash G/K$. We will show that $M_1$ and $M_2$ contain the same nonempty set of proper immersed totally geodesic surfaces. 

Suppose that $S$ is a totally geodesic surface in $M_1$ and denote by $\Gamma$ the orbifold fundamental group of $S$. Then $\Gamma\subset \Gamma_{\mathcal O_1}$. We will show that $\Gamma\subset \Gamma_{\mathcal O_2}$. Let $B_\Gamma$ denote the invariant quaternion algebra of $\Gamma$ and $F$ denote invariant trace field of $\Gamma$. The field $F$ is a totally real proper subfield of $K_1$, hence $F=\mathbb Q$ because $K_1$ is primitive. As it is necessarily the case that $B_\Gamma\otimes_F K_1\cong A_1$, we conclude by Theorem \ref{maxorders} the intersection $\mathcal O_{B_\Gamma}=\mathcal O_1 \cap {B_\Gamma}$ is a maximal order of ${B_\Gamma}$ for which $\mathcal O_{B_\Gamma}\otimes_{\mathbb Z} \mathcal O_{K_1} \cong \mathcal O_1$. It follows that $\Gamma \subset \Gamma_{\mathcal O_{B_\Gamma}} \subset \Gamma_{\mathcal O_1}$. 

We now show that (up to isomorphism) $\Gamma$ is contained in $\Gamma_{\mathcal O_2}$. Proposition \ref{samealgebras} shows that $B_\Gamma\otimes_F K_2\cong A_2$. Identifying $B_\Gamma$ with its image in $A_2$ we have that the maximal order $\mathcal O_{B_\Gamma}$ is contained in the order $\mathcal O_{B_\Gamma}\otimes_{\mathbb Z} \mathcal O_{K_2}$ of $A_2$. Because every maximal order of $K_2$ is conjugate (by Theorem \ref{typenumber}), we conclude that (perhaps upon conjugating $\mathcal O_{B_\Gamma}$) we have $\mathcal O_{B_\Gamma} \subset \mathcal O_2$, hence $\Gamma_{\mathcal O_{B_\Gamma}} \subset \Gamma_{\mathcal O_2}$. As $\Gamma \subset \Gamma_{\mathcal O_{B_\Gamma}}$, this proves that $\Gamma\subset \Gamma_{\mathcal O_2}$. This shows that every proper immersed totally geodesic surface of $M_1$ is also a proper immersed totally geodesic surface of $M_2$. Repeating the same argument with the roles of $M_1$ and $M_2$ reversed completes the proof of the theorem.

\end{proof}


\end{document}